\documentclass[12pt]{article}
\RequirePackage{amsthm,amsmath,amsfonts,mathrsfs,amssymb,hyperref}
\usepackage{xcolor}
\usepackage{graphicx,stmaryrd,enumitem}
\usepackage{dsfont,cite}
\newcommand{\ind}{1\hspace{-0.098cm}\mathrm{l}}

\def\dd{\mbox{d}}

\def\P{\mathbb{P}}
\def\R{\mathbb{R}}

\newtheorem{thm}{Theorem}

\newtheorem{cor}{Corollary}
\newtheorem{lem}{Lemma}
\theoremstyle{remark}
\newtheorem{rem}{Remark}

\theoremstyle{definition}
\newtheorem{defn}{Definition}

\def\deq{\mathrel{\stackrel{d}{=}}} 

\renewcommand{\geq}{\geqslant}
\renewcommand{\leq}{\leqslant}
\usepackage{tikz,authblk}

\def\br{\mathbf{r}}
\def\bn{\mathbf{n}}

\title{Persistence probabilities for MA(1) sequences with uniform innovations}
\author[1]{Frank Aurzada}
\author[2,3]{Kilian Raschel}
\affil[1]{Technical University of Darmstadt, Germany \href{mailto:aurzada@mathematik.tu-darmstadt.de}{\texttt{aurzada@mathematik.tu-darmstadt.de}}}
\affil[2]{CNRS, Université d'Angers, France, \href{mailto:raschel@math.cnrs.fr}{\texttt{raschel@math.cnrs.fr}}}
\affil[3]{International Research Laboratory France-Vietnam in Mathematics and its Applications, CNRS - VAST - VIASM}
\date{\today}

\begin{document}
\maketitle
\allowdisplaybreaks

\begin{abstract}
We study the persistence probabilities of a moving average process of order one with uniform innovations. We identify a number of regions—characterized by the location of the uniform distribution and the coupling parameter of the process—where the persistence probabilities have qualitatively different behaviour. We obtain the generating functions of the persistence probabilities explicitly in all possible regions. In some of the regions, the persistence probabilities can be expressed explicitly in terms of various combinatorial quantities.
\end{abstract}


\section{Introduction and main results}
In recent years, there has been a lot of interest in the study of the persistence probabilities of various stochastic processes, many of them with relations to different physical systems. The persistence probability related to a one-dimensional stochastic process is simply the probability that the process does not change sign in a finite number of steps. In physical terms, the persistence probabilities (and in particular their rate of decay) measure how fast an underlying physical system started in a non-equilibrium state returns to equilibrium. We refer to the survey papers \cite{Bray2013Persistence,SALCEDOSANZ20221} and the monograph \cite{Metzler2014FirstPassage} for an overview on the relevance of the question to physical systems and to \cite{Aurzada2015Persistence} for a survey of the mathematical literature.

In the present paper, we aim to compute the persistence probabilities of a moving average process of order one with uniform innovations. More concretely, for $n\geq 1$ and a coupling parameter $\theta\in\R$, let us look at the probability
\begin{equation}
\label{eq:def_PP}
   p_{n}^a(\theta)  := \mathbb P(X_2\geq \theta X_1,X_3\geq \theta X_2,\ldots,X_{n+1}\geq \theta X_{n}),
\end{equation}
when the $(X_i)$ are i.i.d.\ uniform on $[-a,1]$ with either $a\geq 0$ or $a\in(-1,0]$ (with the convenient abbreviation $p_0^a(\theta):=1$). Let us mention that fixing the right end-point of the uniform distribution at $1$ is w.l.o.g., so that we treat all possible uniform distributions. Fixing the right end point is for notational simplicity.

We note that \eqref{eq:def_PP} is the persistence (i.e.\ non-negativity) probability for the MA(1) process $Z_n:=X_n-\theta X_{n-1}$. The present paper continues the study initiated in \cite{MaDh-01} where the symmetric case $a=1$ is treated.

Another way of looking at the problem is as follows: Consider the $n+1$-dimensional cube $[-a,1]^{n+1}$. We cut this cube along the hyperplanes $\theta x_i = x_{i+1}$, $i=1,\ldots,n$, and keep only the part with $\theta x_i < x_{i+1}$ for all $i$. The question we treat here is: How large is the volume of the remaining polytope (relative to the volume $(1+a)^{n+1}$), depending on $a$ and $\theta$?

Before we can state our main result, one more notation is needed: Define for $z\in\mathbb C$ and $r\in\mathbb C$ with $|r|\leq 1$ the deformed exponential function by
$$
    E(r,z):=\sum_{n=0}^\infty \frac{r^{ n(n-1)/2}}{n!}z^n.
$$
Given this notation, we can already present our main results. Depending on the parameters $a\in(-1,\infty)$ and $\theta\in\R$, the persistence probabilities have a qualitatively different form. The different regions of the parameters with similar structure are depicted in Figure~\ref{fig:map}. We now summarize the results. Their proofs and further refinements and corollaries can be found in the subsequent sections.

The results give the explicit generating function of the persistence probabilities $(p_n^a(\theta))$ in the different regions.

\begin{thm} \label{thm:main}
\begin{enumerate}[label={\rm (\Alph{*})},ref={\rm (\Alph{*})}]
\setcounter{enumi}{1}
    \item\label{region-B}
    Let either $\theta\in[0,1]$ and $a\geq 0$ or $\theta\in[-1,0]$ and $a\in[-\theta,-\frac{1}{\theta}]$.
Then
\begin{equation}  \label{eqn:genfunctiongenaprima}
\sum_{n=0}^\infty  p_{n}^a(\theta) z^n
= \frac{E\left(\theta,\frac{a z}{1+a}\right) -E\left(\theta,\frac{-z}{1+a}\right) }{z  E\left( \theta,\frac{-z}{1+a}\right)}.
\end{equation}
\setcounter{enumi}{6}
\item\label{region-G}
Let $\theta\in[-1,0]$ and $a\geq-1/\theta$. Then
\begin{equation} \label{eqn:greengfprime}
 \sum_{n=0}^\infty p_n^a(\theta) z^n = \frac{\theta a+1}{\theta(1+a)}+ \frac{ E(\theta,\frac{-z}{\theta(1+a)})-E(\theta,\frac{-z}{1+a})  }{zE(\theta,\frac{-z}{1+a})}.
     \end{equation}
\setcounter{enumi}{24}
\item\label{region-Y}
Let $\theta\in[-1,0]$ and $a\in[0,-\theta]$. Then
\begin{equation}
   \label{eq:formuma_pn_a<1_-1<theta<0_bisprime}
    \sum_{n=0}^\infty p_n^a(\theta)z^n=
    \frac{1}{z}\left(\frac{1}{\frac{E(\theta, \frac{az}{\theta(1 + a)})}{E(\theta, \frac{az}{1 + a})}-\frac{1+\frac{a}{\theta}}{1+a} z}-1\right).
\end{equation}

     \setcounter{enumi}{14}
     \item\label{region-O}
     Let $\theta\in(0,1]$ and $-\theta\leq a\leq 0$. Then, 
     \begin{equation}  \label{eqn:darkorangegfprime}
      \sum_{n=0}^\infty p_n^a(\theta) z^n=\frac{F(z)}{1-z F(z)},
\end{equation}
where, using $x_+:=\max(x,0)$,
\begin{align} \label{eqn:darkorangedefnf0}
    F(z):= \sum_{i=0}^\infty (-1)^i  \frac{(\theta^i+a)_+^{i+1}}{(i+1)!\theta^{i(i+1)/2}(1+a)^{i+1}}\, z^i.
\end{align}

\setcounter{enumi}{22}
\item\label{region-W} Let $\theta\in[-1,1]$ and $-1<a\leq (-\theta)\wedge 0$. Then $p_n^a(\theta)=1$ for all $n\geq 0$.
\end{enumerate}
\end{thm}

\begin{figure}[ht!]
\begin{center}
\begin{tikzpicture}[scale=0.8]
\filldraw[fill=blue!50!white, draw=blue!50!white] (-4,4) -- (0,0) -- (0,4) -- cycle; 
\filldraw[fill=blue!50!white, draw=blue!50!white] (0,4) -- (0,0) -- (4,0) -- (4,4) -- cycle; 

\node [black] at (2,2) {Eqn.\ \eqref{eqn:genfunctiongenaprima}};

\filldraw[fill=blue!30!white, draw=blue!30!white] (4,0) rectangle (8,4); 
\node [black] at (6.0,2) {duality~\eqref{eq:duality_theta/pos}};

\filldraw[fill=blue!50!white, draw=blue!50!white] (0,4) rectangle (4,8); 

\node [black] at (2,6) {Eqn.\ \eqref{eqn:genfunctiongenaprima}};

\filldraw[fill=blue!30!white, draw=blue!30!white] (4,4) rectangle (8,8);
\node [black] at (6,6) {duality~\eqref{eq:duality_theta/pos}};

\fill[blue!50!white]  plot[domain=-4:-2, samples=100] (\x, {-16/\x})   -- (-2,4) -- (-4,4) -- cycle;
\fill[blue!50!white]
   (0,4)  -- (0,8)  -- (-2.1,8)  -- (-2.1,4)  -- cycle;

\node [black] at (-1.3,3.5) {Eqn.\ \eqref{eqn:genfunctiongenaprima}};

\fill[green!60!white]
  plot[domain=-4:-2, samples=100] (\x, {-16/\x})
  -- (-2, 8) -- (-4, 8) -- cycle;
\node [black] at (-3,7.5) {Eqn.\ \eqref{eqn:greengfprime}};

\fill[blue!30!white]
  (-8,4)
  -- (-8,2)
  plot[domain=-8:-4, samples=100] (\x, {-16/\x})
  -- (-4,4)
  -- (-8,4)
  -- cycle;

\filldraw[fill=yellow!30!white, draw=yellow!30!white] (-4,4) -- (-5,3.2) -- (-6,2.67) -- (-7,2.29) -- (-8,2) -- (-8,0) -- (-4,0) -- cycle;

\filldraw[fill=yellow!60!white, draw=yellow!60!white] (-4,4) -- (0,0) -- (-4,0) -- cycle;

\node [black] at (-2.6,0.7) {Eqn.\ \eqref{eq:formuma_pn_a<1_-1<theta<0_bisprime}};

\filldraw[fill=orange!80!white, draw=orange!30!white] (0,0) -- (4,-4) -- (4,0) -- cycle;

\node [black] at (2.55,-1) {Eqn.\ \eqref{eqn:darkorangegfprime}};


\fill[orange!30!white]
  plot[domain=4:8, samples=100] (\x, {-16/\x})
  -- (8,0) -- (4,0) -- cycle;
\node [black] at (5.5,-1) {Eqn.\ \eqref{eqn:greysimpleclaim}};
\draw[-,dashed] (4,0) -- (4,-4);
\draw[<->,thick] (3.5,-2.5) -- (4.5,-2.5);
\node [black] at (4,-3) {\eqref{eqn:hiddenduality}};

\node [black] at (-3,4) {\eqref{eqn:hiddenduality2pic}};
\draw[<->,thick] (-3.5,3) -- (-3.5,5);

\node [black] at (-4,-2) {$p_n^a(\theta)\equiv 1$};
\node [black] at (1.5,-3.3) {$p_n^a(\theta)\equiv 1$};
\node [black] at (6.5,-3.3) {$p_n^a(\theta)\equiv 0$};

\filldraw[fill=blue!30!white, draw=blue!30!white] (-8,8) -- (-8,4) -- (-4,4) -- cycle;
\filldraw[fill=green!30!white, draw=green!40!white] (-8,8) -- (-4,8) -- (-4,4) -- cycle;

\draw[<->,thick] (3,5) -- (5,3);
\draw[<->,thick] (5,5) -- (3,3);
\draw[-,dashed] (4,0) -- (4,8);
\draw[-,dashed] (0,4) -- (8,4);

\draw[<->,thick] (-7,5.5) -- (-1.5,5.5);
\draw[<->,thick] (-5,6) -- (-3,6);
\draw[<->,thick] (-5,2) -- (-3,2);
\draw[<->,thick] (-7,2.5) -- (-1.5,2.5);
\node [black] at (-6.8,5) {duality~\eqref{eq:duality_theta/neg}};
\node [black] at (-6.8,3) {duality~\eqref{eq:duality_theta/neg}};
\node [black] at (-5.3,6.5) {duality~\eqref{eq:duality_theta/neg}};
\node [black] at (-5.3,1.5) {duality~\eqref{eq:duality_theta/neg}};
\draw[-,dashed] (-4,0) -- (-4,8);
\draw[-,dashed] (0,4) -- (-8,4);

\draw[->,very thick] (-9,0) -- (9,0);
\draw[->,very thick] (0,-4) -- (0,9);
\node [black] at (9.5,0) {\large $\theta$};
\node [black] at (4,-0.5) {\large $1$};
\node [black] at (-4,-0.5) {\large $-1$};
\node [black] at (0,9.5) {\large $a$};
\node [black] at (0.3,4.4) {\large $1$};
\node [black] at (0.5,-0.5) {\large $0$};
\node [black] at (-0.7,-3.7) {\large $-1$};
\draw[-,thick] (-9,-4) -- (9,-4);
\end{tikzpicture}
\end{center}
\caption{Phase diagram of the main results}
\label{fig:map}
\end{figure}

It is worth noting that, in the study of persistence problems, explicit and comprehensive results like Theorem~\ref{thm:main} are rather rare; most works instead concentrate on asymptotic behavior, often through the persistence exponent.

\begin{rem} \label{rem:mainthm}
A few comments on these results are due:
\begin{enumerate}[label=(\alph*)]
\item We marked the different regions with colours for a better overview. Correspondingly, the letters in  Theorem~\ref{thm:main} stand for blue \ref{region-B}, green \ref{region-G}, yellow \ref{region-Y},   orange \ref{region-O}, and white \ref{region-W}, cf.\ Figure~\ref{fig:map}.
\item The theorem only considers the case $\theta\in[-1,1]$. The persistence probabilities for $|\theta|>1$ can be obtained immediately from Theorem~\ref{thm:main} and duality formulas given in Lemma~\ref{lem:duality1} (treating all $a\geq 0$) and Lemma~\ref{lem:duality2} (treating in particular $a\in(-1,0]$). The dualities are of interest in their own right.
\item \label{rem:itempowerseries} It is reasonable to ask for which $z$ the statements in Theorem~\ref{thm:main} hold. This is to be understood as follows: Let $z_0$ be the singularity of the function on the right-hand side with smallest modulus. Then the quantity on the left-hand side, i.e.\ the generating function of the $(p_n^a(\theta))$, is finite for all $z$ with $|z|<|z_0|$ and both quantities coincide in this case. This also concerns the rate of decay of the persistence probability,  cf.\ item~\ref{item:remarkmainpe} below. 
Alternatively, the equalities stated in Theorem~\ref{thm:main} may also be interpreted as identities between formal power series.
\item The blue, green, and yellow regions are closely interconnected. More precisely, the formula in the green region can be obtained from formula in blue region through a conditioning argument. And, by a duality formula in Corollary~\ref{cor:second_hidden_duality}, the formula in the green region is equivalent to the formula in the yellow region.
\item For $\theta=1$ and any $a>-1$, we have $p_n^a(1)=\frac{1}{(n+1)!}$, as can be seen from the borderline of the blue region \ref{region-B} and the orange region \ref{region-O}, respectively. This result is due to \cite{MaDh-01} and \cite{Krishna2016Persistence} using different proofs. We give a proof here in Lemma~\ref{lem:trivialcases} below in order to make this paper fully self-contained.
\item We further remark that the result for the symmetric case $a=1$ was obtained already in \cite{MaDh-01}, see their equations (35) and (36).
\item \label{item:remarkmainpe} Of great interest in the study of persistence probabilities is typically the exponential decay rate, also called persistence exponent. In all cases of Theorem~\ref{thm:main}, the persistence exponent $\lambda$ can be obtained as the smallest real, positive zero $z_0=:1/\lambda$ of the de\-no\-mi\-na\-tor of the generating function. For example, in the blue region, let $z_0=:1/\lambda$ be the smallest real, positive zero $z$ of the equation $E(\theta,\frac{-z}{1+a})=0$. Then one can show from Theorem~\ref{thm:main} that
$$
p_n^a(\theta) \sim \frac{(1+a) E(\theta,\frac{a}{\lambda(1+a)})}{E(\theta,\frac{-\theta}{\lambda(1+a)})} \cdot \lambda^{n+2},\qquad \text{as $n\to\infty$}.
$$
This is studied in great detail in all regions in \cite{AurzadaWorf2026conditioned}.
\item Note that, in the orange region \ref{region-O}, the generating function of the persistence probabilities \eqref{eqn:darkorangegfprime} is piecewise rational; in the light orange region ($\theta>1$, $a\in[-1/\theta,0)$), it is even piecewise polynomial. In contrast, in the blue, green, and yellow region, respectively, the generating function consists of a combination of the (transcendent)\ deformed exponential functions $E(.,.)$ with different arguments. 

This type of classification of a probabilistic problem, based on the algebraic nature (rational vs.\ transcendental) of the generating function, recalls a combinatorial model that has been extensively studied in the literature over the past two decades: walks in the quarter plane; see \cite{BMMi-10}. In that context, the central question was to classify the transition weights according to whether the counting generating function is rational, algebraic, or transcendental. We emphasize that there is no direct formal connection between quarter plane walks and the model considered in the present paper. The parallel lies instead in the underlying classification problem: in both cases, one seeks to understand how the parameters of the model determine the analytic and algebraic nature of its generating function. Such classifications are of intrinsic interest, but they may also shed light on the complexity of the model and on the existence of explicit formulas, precise asymptotic estimates, or qualitatively different regimes of parameter dependence. 
\item In the blue region, see \eqref{eqn:genfunctiongenaprima}, and in the green region, see \eqref{eqn:greengfprime}, one can express the persistence probabilities $p_n^a(\theta)$ explicitly in terms of purely combinatorial quantities, see Corollaries~\ref{cor:combinatorialformulaa0theta01}, \ref{cor:recurrencezeropositivethetaless1}, \ref{cor:largethetatosmallthetaa}, and \ref{cor:greencombrepresentation}.
\item Let us comment on the complicated landscape of solutions, i.e.\ the different regions. To understand this, consider the two-dimensional problem $x_2\geq \theta x_1$ with $(x_1,x_2)\in[-a,1]^2$. The results are qualitatively different depending on the location of the line $x_2=\theta x_1$ in relation to the box $[-a,1]^2$. In the blue region, the line passes through the box crossing two opposite edges. In the yellow region, the line passes through two adjacent edges of the box and the allowed area ($x_2\geq \theta x_1$) is the larger part of the box. In contrast, in the green region, the line  $x_2=\theta x_1$ passes through two adjacent edges of the box $[-a,1]^2$ and the allowed area is only a small triangle in the upper right corner of the box. In the orange regions, the box $[-a,1]^2$ does not contain the origin and thus creates two other cases depending on whether the line cuts out the larger or the smaller part of the box. This is a visualization of the origin of the different regions.

\begin{figure}[htbp]
  \centering
  \begin{minipage}{0.2\textwidth}
\begin{tikzpicture}[scale=0.8]

\draw[->] (-2.5,0) -- (1.6,0) node[right] {$x_1$};
\draw[->] (0,-2.5) -- (0,1.6) node[above] {$x_2$};

\draw[thick] (-2,-2) rectangle (1,1);

\draw (-2,0.06) -- (-2,-0.06);
\node[below right] at (-2.15,0.15) {$-a$};

\draw (1,0.06) -- (1,-0.06);
\node[below right] at (1,0.65) {$1$};

\draw (0.06,-2) -- (-0.06,-2);
\node[below left] at (0,-2) {$-a$};

\draw (0.06,1) -- (-0.06,1);
\node[above left] at (0,1) {$1$};

\draw[thick,blue!70!white] (-2.3,-1.15) -- (1.3,0.65);
\node[above] at (-0.2,-1.2) {$x_2=\theta x_1$};
\draw[thick,blue!70!white] (-2.3,0.575) -- (1.3,-0.325);

\end{tikzpicture}
  \end{minipage}
  \hfill
\begin{minipage}{0.2\textwidth}
\begin{tikzpicture}[scale=1.3]

\draw[->] (-1.0,0) -- (1.6,0) node[right] {$x_1$};
\draw[->] (0,-1) -- (0,1.7) node[above] {$x_2$};

\draw[thick] (-0.5,-0.5) rectangle (1,1);

\draw (-0.5,0.06) -- (-0.5,-0.06);
\node[below right] at (-1.2,0) {$-a$};

\draw (1,0.06) -- (1,-0.06);
\node[below right] at (1,0) {$1$};

\draw (0.06,-0.5) -- (-0.06,-0.5);
\node[below left] at (0,-0.5) {$-a$};

\draw (0.06,1) -- (-0.06,1);
\node[above left] at (0,1) {$1$};

\draw[thick,yellow!60!white] (-0.8,0.5) -- (0.96,-0.6);

\node[right] at (-0.45,0.5) {$x_2=\theta x_1$};

\end{tikzpicture}
\end{minipage}
  \hfill
  \begin{minipage}{0.2\textwidth}

\begin{tikzpicture}[scale=0.8]

\draw[->] (-2.5,0) -- (1.6,0) node[right] {$x_1$};
\draw[->] (0,-2.5) -- (0,1.6) node[above] {$x_2$};

\draw[thick] (-2,-2) rectangle (1,1);

\draw (-2,0.06) -- (-2,-0.06);
\node[below right] at (-2.15,0.15) {$-a$};

\draw (1,0.06) -- (1,-0.06);
\node[below right] at (1,0) {$1$};

\draw (0.06,-2) -- (-0.06,-2);
\node[below left] at (0,-2) {$-a$};

\draw (0.06,1) -- (-0.06,1);
\node[above left] at (0,1) {$1$};

\draw[thick,green!80!white] (-1.5,1.3) -- (1.3,-1.12);

\node at (-0.2,-1.5) {$x_2=\theta x_1$};

\end{tikzpicture}

  \end{minipage}
\hfill
  \begin{minipage}{0.2\textwidth}
\begin{tikzpicture}[scale=1.8]

\draw[->] (-0.08,0) -- (1.25,0) node[right] {$x_1$};
\draw[->] (0,-0.08) -- (0,1.30) node[above] {$x_2$};

\draw[thick] (0.5,0.5) rectangle (1,1);

\draw (0.5,0.04) -- (0.5,-0.04);
\node[below] at (0.5,0) {$-a$};

\draw (1,0.04) -- (1,-0.04);
\node[below right] at (1,0) {$1$};

\draw (0.04,0.5) -- (-0.04,0.5);
\node[left] at (0,0.5) {$-a$};

\draw (0.04,1) -- (-0.04,1);
\node[above left] at (0,1) {$1$};

\draw[thick,orange!60!white] (0,0) -- (1.2,0.8);
\node[right] at (0.5,0.3) {$x_2=\theta x_1$};

\end{tikzpicture}  \end{minipage}
 \caption{Visualization of the relation of the line $x_2=\theta x_1$ to the box $[-a,1]^2$ in the blue, yellow, green, and orange region, respectively.}
  \end{figure}
\end{enumerate}
\end{rem}

\paragraph{Related work.} Let us mention some related papers. We already cited \cite{MaDh-01} that deals with the symmetric case $a=1$ and which was the starting point of our work. On the other hand, a follow-up work is \cite{AurzadaWorf2026conditioned}, which uses the results from the present paper to obtain explicit Doob transforms of MA(1) sequences with uniform innovations conditioned to be positive.

In \cite{AuMuZe-21}, general Markov chains and their non-exit probabilities from sets are studied. A particular case is given by moving-average processes (with general innovation distribution and general order), for which the persistence exponent is characterized as the leading eigenvalue of an eigenvalue equation. This eigenvalue equation is studied with perturbation techniques in \cite{aurzadabothe} for the standard normal innovation distribution. Further related works are \cite{FeldheimFeldheim2015,Krishna2016Persistence,FeldheimFeldheimNitzan2021,FeldheimFeldheimMukherjee2025}, where estimates are given for the Gaussian case and the result of the present paper for the special case $\theta=1$ appears. See \cite{GoLu-20} for an analysis of the records for the moving average of a time series, in particular for uniform innovations.

We further mention \cite{AlBoRaSi-23}, which treats the persistence probabilities of order one autoregressive processes with uniform innovations using similar techniques. The same question for autoregressive processes is treated in \cite{larralde04,novikov08a,novikov08b,baumgarten14,kettner19,HiKoWa-20,AuMuZe-21,dehikowa-22,VyWa-23} with different methods. Finally, we mention that we deal with zeros of the deformed exponential function in some cases, which are also an interesting object of study in its own right, see e.g.\ \cite{Sokal-talk,WaZh-18,Ku-24} and references therein.

\paragraph{Generalisations.} Clearly, it is desirable to obtain similar results for other innovation distributions. In this context, let us mention  \cite{AurzadaRaschel2025a} which treats the case of possibly unsymmetric Laplace innovations. Another line of research is higher order moving average sequences. We believe that the generating function approach can be very useful also in the mentioned related and generalised setups. A path for applying the present results is conditioning the MA(1) sequence to be positive via a Doob transform, which is the purpose of \cite{AurzadaWorf2026conditioned,AurzadaWorf2026Laplace}. This requires a careful analysis of the decay rate of the persistence probabilities; in particular, it only works in the standard way when the persistence probabilities decay with an exponential rate (which is classified in Proposition 2.3 in \cite{AuMuZe-21} for general moving average processes).

\paragraph{Outline.} The structure of this paper is as follows. In Section~\ref{sec:dualities}, we formulate and prove four crucial dualities, which in particular allow to compute the persistence probabilities $p_n^a(\theta)$ in terms of $p_n^a(1/\theta)$ or $p_n^{1/a}(1/\theta)$ or $p_n^{1/a}(\theta)$, respectively. Section~\ref{sec:combinatorics} is devoted to certain combinatorial quantities that describe the series representation of $1/E(\theta,z)$ and that are used later to express the $p_n^a(\theta)$ explicitly. Finally, Sections~\ref{sec:blue}, \ref{sec:green}, \ref{sec:yellow}, and \ref{sec:orange} contain the proofs of Theorem~\ref{thm:main} in the blue, green, yellow, and orange region, respectively. Furthermore, a number of additional properties and refinements are given in those sections; in particular, the persistence probabilities can be rewritten in terms of the combinatorial quantities from Section~\ref{sec:combinatorics} or as Mallows-Riordan polynomials in some cases.

\section{Dualities and trivial cases} \label{sec:dualities}

We start with four crucial dualities that allow to reduce some cases to others.

\begin{lem} \label{lem:duality1} Let $a>0$. Then for $\theta>0$, we have
\begin{equation} 
\label{eq:duality_theta/pos}
 p_n^a(\theta)=p_n^{1/a}(1/\theta).
\end{equation}
For $\theta<0$, we have 
\begin{equation}
    \label{eq:duality_theta/neg}
    p_n^a(\theta) = p_n^a(1/\theta).
\end{equation}
\end{lem}

\begin{proof}   Throughout, we will use $\bar\theta:=1/\theta$. Let $\theta>0$.  Then 
    \begin{align*}
        p_n^a(\theta) & = \P(X_2\geq \theta X_1,X_3\geq \theta X_2,\ldots,X_{n+1}\geq \theta X_{n})
        \\
        & = \P(\bar\theta X_2\geq X_1, \bar\theta X_3\geq  X_2,\ldots,\bar\theta X_{n+1}\geq X_{n})
        \\
        & = \P(\bar\theta X_n \geq X_{n+1}, \bar\theta X_{n-1}\geq  X_n,\ldots,\bar\theta X_{1}\geq X_{2})
        \\
        & = \P(\bar\theta X_{1}\geq X_{2},\ldots,\bar\theta X_{n-1}\geq  X_n,\bar\theta X_n \geq X_{n+1})
        \\
        & = \P(\bar\theta\,(-X_{1}/a)\leq -X_{2}/a,\ldots,\bar\theta\,(-X_{n-1}/a)\leq -X_n/a,\bar\theta\,(-X_n/a )\leq -X_{n+1}/a)
        \\
        &= p_n^{1/a}(1/\theta).
        \end{align*}
    In the last computation, we used that $\theta>0$ (second step), the fact that $(X_1,\ldots,X_{n+1})\deq (X_{n+1},\ldots,X_1)$ (third step),  the assumption that $a>0$ (fifth step), and the fact that the $-X_i/a$ are uniformly distributed on $[-1/a,1]$ (last step).

    Let $\theta<0$. Then, with a similar, but even simpler computation, we obtain: 
    \begin{align*}
        p_n^a(\theta) & = \P(X_2\geq \theta X_1,X_3\geq \theta X_2,\ldots,X_{n+1}\geq \theta X_{n})
        \\
        & = \P(\bar\theta X_2\leq X_1, \bar\theta X_3\leq  X_2,\ldots,\bar\theta X_{n+1}\leq X_{n})
        \\
        & = \P(\bar\theta X_n \leq X_{n+1}, \bar\theta X_{n-1}\leq  X_n,\ldots,\bar\theta X_{1}\leq X_{2})
        \\
        & = \P(\bar\theta X_{1}\leq X_{2},\ldots,\bar\theta X_{n-1}\leq  X_n,\bar\theta X_n \leq X_{n+1})
        \\
        &= p_n^{a}(1/\theta).
        \end{align*}
    We used that $\theta<0$ (second step) and the fact that $(X_1,\ldots,X_{n+1})\deq (X_{n+1},\ldots,X_1)$ (third step).
\end{proof}

 We stress that the next duality is valid for general continuous distributions.

\begin{lem} \label{lem:duality2} 
Let $(X_i)$ be i.i.d.\ random variables with continuous distribution and $\theta>0$. Set $p_0(\theta):=1$ and, for $n\geq 1$, $p_n(\theta)$ as in \eqref{eq:def_PP}. Then
\begin{equation} \label{eqn:hiddenduality}
    p_n(\theta)  = \sum_{k=1}^n p_{n-k}(\theta) \cdot p_{k-1}(1/\theta) (-1)^{k-1} + (-1)^{n}p_n(1/\theta),\qquad n\geq 0.
\end{equation}
Furthermore, denoting $\widehat P_\theta(z):=\sum_{n=0}^\infty  p_n(\theta)z^n$, we have
\begin{align}
 \widehat P_\theta(z) & = \frac{\widehat P_{1/\theta}(-z)}{1- z \widehat P_{1/\theta}(-z)}.
 \label{eqn:hiddendualityGF}
\end{align}
\end{lem}

\begin{proof} Let $\bar\theta:=1/\theta$.
    Set $A_i:=\{\theta X_i \leq X_{i+1}\}$ and $B_i:=\{X_i \geq \bar\theta X_{i+1}\}$. Notice that $B_i$ is the complement of $A_i$ up to a set of probability zero, because the $X_i$ have a continuous distribution. Using this for $A_n$/$B_n$ in the first step and for $A_{n-1}$/$B_{n-1}$ in the second step, we obtain
    \begin{align*}
        p_n(\theta) & = \P( A_1 \cap \ldots \cap A_n )
        \\
        & = \P( A_1 \cap \ldots \cap A_{n-1} ) -  \P( A_1 \cap \ldots \cap A_{n-1} \cap B_n)
        \\
        & = \P( A_1 \cap \ldots \cap A_{n-1} ) -  (\P( A_1 \cap \ldots \cap A_{n-2} \cap B_n)-\P( A_1 \cap \ldots \cap A_{n-2}\cap B_{n-1} \cap B_n))
                \\
        & = \P( A_1 \cap \ldots \cap A_{n-1} ) -  \P( A_1 \cap \ldots \cap A_{n-2} ) \cdot \P(B_n)+\P( A_1 \cap \ldots \cap A_{n-2}\cap B_{n-1} \cap B_n)
                \\
        & = p_{n-1}(\theta) -  p_{n-2}(\theta) \cdot p_1(\bar\theta) + \P( A_1 \cap \ldots \cap A_{n-3}\cap B_{n-1} \cap B_n)\\
         & \qquad \quad -\P( A_1 \cap \ldots \cap A_{n-3}\cap B_{n-2}\cap B_{n-1} \cap B_n),
    \end{align*}
    where we used in the third step that $A_1 \cap \ldots \cap A_{n-2}$ and $B_n$ are independent and, by exchangeability,
    \begin{align} \label{eqn:onlythingtoreplacehiddenduality2}
    \P( B_i \cap \ldots \cap B_n ) & = \P( X_i \geq \bar\theta X_{i+1}, \ldots,   X_n \geq \bar\theta X_{n+1})
    \\
    & = \P( X_{n-i+2} \geq \bar\theta X_{n-i+1}, \ldots,   X_{2} \geq \bar\theta X_{1}) = p_{n-i+1}(\bar\theta).\nonumber
    \end{align}
    Continuing this way, we obtain that
    \begin{align*}
        p_n(\theta)
        & = \sum_{k=1}^n (-1)^{k-1} p_{n-k}(\theta) \cdot p_{k-1}(\bar\theta) + (-1)^n p_n(\bar\theta),
    \end{align*}
    which is \eqref{eqn:hiddenduality}.
    Multiplying by $z^n$ and summing in $n$, we obtain
        \begin{align*}
   \widehat P_\theta(z) - 1 & = \sum_{n=1}^\infty  p_n(\theta)z^n
   \\
   & = \sum_{n=1}^\infty \sum_{k=1}^n p_{n-k}(\theta) z^{n-k} \cdot z p_{k-1}(\bar\theta) (-1)^{k-1} z^{k-1} +  \sum_{n=1}^\infty (-1)^n  p_n(\bar\theta)z^n
    \\
   & = \widehat P_{\theta}(z) \cdot z \widehat P_{\bar\theta}(-z) + \widehat P_{\bar\theta}(-z) - 1,
        \end{align*}
        as claimed in \eqref{eqn:hiddendualityGF}.
\end{proof}

Also the next lemma is valid for general continuous distributions.

\begin{lem} \label{lem:secondhiddenduality}
Let $(X_i)$ be i.i.d.\ random variables with continuous distribution and $\theta<0$. Set $p_0(\theta):=1$ and, for $n\geq 1$, $p_n(\theta)$ as in \eqref{eq:def_PP}, and
\begin{align*}
q_n(\theta) & :=\P( \theta X_1 \geq X_2, \ldots, \theta X_n \geq X_{n+1} ). 
\end{align*}
Then
\begin{equation*} 
    p_n(\theta)  = \sum_{k=1}^n p_{n-k}(\theta) \cdot q_{k-1}(1/\theta) (-1)^{k-1} + (-1)^{n}q_n(1/\theta),\qquad n\geq 0.
\end{equation*}
Furthermore, denoting $\widehat P_\theta(z):=\sum_{n=0}^\infty  p_n(\theta)z^n$ and $\widehat Q_\theta(z):=\sum_{n=0}^\infty  q_n(\theta)z^n$, we have
\begin{equation}
\label{eqn:hiddendualityGF2}
 \widehat P_\theta(z)  = \frac{\widehat Q_{1/\theta}(-z)}{1- z \widehat Q_{1/\theta}(-z)}.
\end{equation}
\end{lem}

\begin{proof}
    We proceed exactly as in the proof of Lemma~\ref{lem:duality2}. Set $\bar\theta:=1/\theta$. This time, since $\theta<0$, $A_i:=\{ \theta X_i \leq X_{i+1}\}=\{ X_i \geq \bar \theta X_{i+1}\}$, $B_i:=\{ X_i \leq \bar\theta X_{i+1}\}$. By the assumption that the $X_i$ have a continuous distribution, $B_i$ is the complement of $A_i$ up to a set of probability zero. We can now follow line-by-line the proof of Lemma~\ref{lem:duality2}. The only difference is that this time, instead of \eqref{eqn:onlythingtoreplacehiddenduality2}, we have
        \begin{align*}
    \P( B_i \cap \ldots \cap B_n ) & = \P( X_i \leq \bar\theta X_{i+1}, \ldots,   X_n \leq \bar\theta X_{n+1})
    \\
    & = \P( X_{n-i+2} \leq \bar\theta X_{n-i+1}, \ldots,   X_{2} \leq \bar\theta X_{1}) = q_{n-i+1}(\bar\theta),
    \end{align*}thereby completing the proof.
\end{proof}

Let us specify the last result for uniform distributions.

\begin{cor}
\label{cor:second_hidden_duality}
Let $\theta<0$ and $a>0$. Set $\widehat P_{a,\theta}(z) := \sum_{n=0}^\infty p_n^a(\theta) z^n$. Then
\begin{equation} \label{eqn:hiddenduality2pic}
    \widehat P_{a,\theta}(z)  = \frac{\widehat P_{1/a,\theta}(-z)}{1-z\widehat P_{1/a,\theta}(-z)}.
\end{equation}
\end{cor}

\begin{proof} We use Lemma~\ref{lem:secondhiddenduality} and observe that
\begin{align*}
q_n(\bar\theta) & =\P( \bar\theta X_1 \geq X_2, \ldots, \bar\theta X_n \geq X_{n+1} )
\\
&= \P( \bar\theta (-X_1/a) \leq (-X_2/a), \ldots, \bar\theta (-X_n/a) \leq (-X_{n+1}/a) )
\\
& = p_n^{1/a}(\bar \theta) = p_n^{1/a}(\theta),
\end{align*}
where we used that the $-X_i/a$ are uniform on $[-1/a,1]$ (in the third step) and the duality \eqref{eq:duality_theta/neg} in the last step. Therefore, the generating function on the right-hand side in \eqref{eqn:hiddendualityGF2} becomes $\widehat P_{1/a,\theta}$, while the generating function on the left-hand side in \eqref{eqn:hiddendualityGF2} is $\widehat P_{a,\theta}$, as claimed.    
\end{proof}

After treating the crucial dualities, we mention a few trivial cases. Note that Lemma~\ref{lem:trivialcases}\ref{item-a}\ref{item-i} proves the result claimed in the white case of Theorem~\ref{thm:main}.

\begin{lem} \label{lem:trivialcases}
\begin{enumerate}[label={\rm (\alph{*})},ref={\rm (\alph{*})}]
    \item\label{item-a}
Let $a\in(-1,0]$.
\begin{enumerate}[label={\rm (\roman{*})},ref={\rm (\roman{*})}]
        \item\label{item-i} For $\theta\leq -a$ we have $p_n^a(\theta)=1$ for all $n\geq 0$.
        \item\label{item-ii} For $\theta\geq -1/a$ we have $p_n^a(\theta)=0$ for all $n\geq 1$.
    \end{enumerate}
\item\label{item-b}
    For any $a\in(-1,\infty)$ and $\theta=1$ we have 
\begin{align*}
    p_n^a(1) & = \frac{1}{(n+1)!}.
    \end{align*}
\end{enumerate}
\end{lem}

Let us stress that the result in \ref{item-b} for $\theta=1$ (and in fact also its proof)\ is valid for general continuous distributions. It was first proved by \cite{MaDh-01} (see their (28)) and re-proved in \cite{Krishna2016Persistence} (see their Example 6). We mention the proof here for completeness.

\begin{proof}
First, we prove part \ref{item-i} of  \ref{item-a}.
   Let $a\in(-1,0]$ and $0\leq \theta\leq -a$. Then for any $i$
   $$
   \theta X_i \leq \theta \leq -a \leq X_{i+1},
   $$
   so that $p_n^a(\theta)=1$. On the other hand, let $a\in(-1,0]$ and $\theta\leq 0$. Then for any $i$
   $$
   \theta X_i \leq 0 \leq -a \leq X_{i+1},
   $$
   so that $p_n^a(\theta)=1$.
   
Next, we prove part \ref{item-ii} of  \ref{item-a}. If $a\in(-1,0)$ and $\theta\geq -1/a$ we have for any $i$
   $$
   \theta X_i \geq  \frac{-1}{a} \, X_i \geq \frac{-1}{a} \cdot (-a) = 1 \geq  X_{i+1},
   $$
   so that $p_n^a(\theta)=0$.
   
Let us finally show part  \ref{item-b}. Note that for any permutation $\pi$ of the set $\{1,\ldots,n+1\}$, by exchangeability,
$$
\mathbb P(X_{\pi(1)} \leq X_{\pi(2)} \leq \ldots\leq  X_{\pi(n)} \leq X_{\pi(n+1)}) = \mathbb P(X_1 \leq X_2 \leq \ldots \leq X_n \leq X_{n+1}) = p_n^a(1).
$$
Summing this over all $\pi$ and simplifying (and using the fact that the distribution of $(X_i)$ is continuous) immediately implies the claim.\end{proof}

We remark that the statement in Lemma~\ref{lem:trivialcases}\ref{item-a}\ref{item-ii} can be obtained from Lemma~\ref{lem:trivialcases}\ref{item-a}\ref{item-i}, by the duality \eqref{eqn:hiddenduality}. However, we prefer the more direct and simple proof.

\section{A combinatorial lemma} \label{sec:combinatorics}

In some of the cases that we treat, expressions of the type $\frac{1}{E(\theta,z)}$ appear, where $E$ is the deformed exponential function, cf.\ \eqref{eqn:genfunctiongenaprima}, \eqref{eqn:greengfprime}, and \eqref{eq:formuma_pn_a<1_-1<theta<0_bisprime}. The main lemma of this section, Lemma~\ref{lem:comblem}, which may be of independent interest, gives a purely combinatorial description of the terms in the series expansion of $\frac{1}{E(\theta,z)}$ at $z=0$.

For this purpose, we need a notion to describe certain combinatorial quantities.

\begin{defn} Fix an integer $\ell\geq 1$. An $\ell$-profile is a vector $\br=(r_1,\ldots,r_\ell)$ with integers $r_i\geq 0$ and $\sum_{i=1}^\ell i r_i = \ell$.
\end{defn}

An $\ell$-profile records a collection of objects with total size $\ell$, where $r_i$ denotes the number of objects of size $i$; thus $\sum_{i=1}^\ell i r_i=\ell$. Let us mention a few examples: The only $1$-profile is $\br=(1)$. The only two $2$-profiles are $\br=(2,0)$ and $\br=(0,1)$. The only three $3$-profiles are $\br=(3,0,0)$, $\br=(1,1,0)$, and $\br=(0,0,1)$.

We can now state and prove the main lemma of this section.

\begin{lem}
\label{lem:comblem}
Let $\theta\in\mathbb{C}$ with $|\theta|\leq 1$ and $z\in\mathbb{C}$. Then
\begin{equation*}
  \frac{1}{E(\theta,z)}  = \sum_{\ell=0}^\infty \phi_\ell(\theta) z^\ell,  
\end{equation*}
where $\phi_0(\theta):=1$ and
\begin{equation} \label{eqn:defphi}
\phi_\ell(\theta) := \sum_{\text{$\ell$-profile $\br$}}  \binom{\sum_{i=1}^\ell r_i}{r_1,\ldots,r_\ell}  \prod_{i=1}^\ell \left[\frac{ -\theta^{i(i-1)/2}}{i!} \right]^{r_i}.
\end{equation}
\end{lem}

The statement of the lemma has to be understood in the sense of Remark~\ref{rem:mainthm} \ref{rem:itempowerseries}.

In the proof of the lemma, we shall require the following additional combinatorial quantity.

\begin{defn} Fix $k\geq 1$ and $\ell\geq 1$. For a vector $\bn=(n_1,\ldots,n_k)$ with integers $n_i\geq 0$ and $n_1+\ldots+n_k=\ell$, we define the $\ell$-profile of $\bn$ by
\begin{equation*}
r(\bn):=(r_1(\bn),\ldots,r_\ell(\bn)),\qquad\text{where}\qquad r_i(\bn):=|\{ j\, |\, n_j = i \}| = \sum_{j=1}^k \ind_{n_j=i}.
\end{equation*}
\end{defn}

Note that $k-\sum_{i=1}^\ell r_i(\bn)$ is the number of zeros in the integer vector $(n_1,\ldots,n_k)$. For an $\ell$-profile $\br$, set
$$
w_k(\br):=\sum_{n_1+\ldots+n_k = \ell, n_i\geq 0, r(\bn)=\br} 1.
$$
This is the number of distinct ways to obtain the $\ell$-profile $\br$ by using integer vectors $\bn$ of length $k$.
Furthermore, a simple combinatorial fact is that for any $\ell$-profile $\br=(r_1,\ldots,r_\ell)$ with $k\geq \sum_{i=1}^\ell r_i$,
\begin{equation} \label{eqn:numberwkbr}
w_k(\br) = \binom{k}{r_1,\ldots,r_\ell,k-\sum_{i=1}^\ell r_i} = \frac{k!}{(\prod_{j=1}^\ell r_j! ) \cdot (k-\sum_{i=1}^\ell r_i)!}.
\end{equation}


\begin{proof}[Proof of Lemma~\ref{lem:comblem}] As a first step, we use the geometric series and write $1/E(\theta,z)$ as a series in $z$ with coefficients that consist of multiple sums:
    \begin{align}
 \frac{1}{E(\theta,z)}
       &= 
        \frac{1}{1-\sum_{n=1}^\infty\frac{-\theta^{\frac{n(n-1)}{2}}z^n}{n!}} \nonumber
       \\
       &=1+\sum_{k=1}^\infty \left(\sum_{n=1}^\infty\frac{-\theta^{\frac{n(n-1)}{2}}z^n}{n!}\right)^k\nonumber\\
        &=1+\sum_{k=1}^\infty (-1)^k\sum_{n_1=1}^\infty \cdots \sum_{n_k=1}^\infty \frac{1}{n_1!}\cdots \frac{1}{n_k!}z^{n_1+\cdots +n_k} \theta^{\frac{n_1(n_1-1)}{2}+\cdots +\frac{n_k(n_k-1)}{2}}\nonumber\\
        &= 1+ \sum_{\ell=1}^\infty z^\ell  \sum_{k=1}^\infty(-1)^k\sum_{n_1=1}^\infty \cdots\sum_{n_k=1}^\infty \mathds{1}_{n_1+\cdots+n_k=\ell}\frac{1}{n_1! \ldots n_k!}\theta^{\frac{n_1(n_1-1)}{2}+\cdots +\frac{n_k(n_k-1)}{2}}\nonumber
        \\
                &= 1+ \sum_{\ell=1}^\infty z^\ell  \sum_{k=1}^\ell(-1)^k\sum_{n_1'=1}^\ell \cdots\sum_{n_k'=1}^\ell \mathds{1}_{n_1'+\cdots+n_k'=\ell}\frac{1}{n_1'! \ldots n_k'!}\theta^{\frac{n_1'(n_1'-1)}{2}+\cdots +\frac{n_k'(n_k'-1)}{2}}.\label{eq:1/E_btc}
    \end{align}

We shall continue to compute the coefficients of the series (only). We  use the notation $\kappa(\bn)$ for the number of zeros in the vector $\bn=(n_1,\ldots,n_\ell)$, i.e.\ $\kappa(\bn):=\sum_{i=1}^\ell \ind_{n_i=0}$. Recall that $\kappa(\bn)=\ell - \sum_{i=1}^\ell r_i(\bn)$. In the following computation, the combinatorial factor $\binom{\ell}{\ell-\kappa(\bn)}$ enters in the second step because one can choose the location of the $\ell-\kappa(\bn)=k$ non-zeros in the vector $(n_1,\ldots,n_\ell)$ with $n_1+\ldots+n_\ell=\ell$ and $n_i\geq 0$. Having chosen those gives the vector $(n_1',\ldots,n_k')$ with $n_i'\geq 1$. This allows us to compute the coefficients:
\begin{align*}
 &
\sum_{k=1}^\ell (-1)^k \sum_{n_1'=1}^\ell \cdots \sum_{n_k'=1}^\ell \ind_{n_1'+\ldots+n_k'=\ell} \prod_{j=1}^k \frac{\theta^{n_j'(n_j'-1)2/2}}{n_j'!}
\\
&= \sum_{n_1+\ldots+n_\ell=\ell, n_i\geq 0} \frac{1}{\binom{\ell}{\ell-\kappa(\bn)}} (-1)^{\ell-\kappa(\bn)} \prod_{j=1}^\ell \frac{\theta^{n_j(n_j-1)2/2}}{n_j!}
\\
&=\sum_{\text{$\ell$-profile $\br$}} \sum_{n_1+\ldots+n_\ell=\ell, n_i\geq 0} \ind_{r(\bn)=\br} \frac{1}{\binom{\ell}{\ell-\kappa(\bn)}} (-1)^{\ell-\kappa(\bn)} \prod_{j=1}^\ell \frac{\theta^{n_j(n_j-1)2/2}}{n_j!}
\\
&= \sum_{\text{$\ell$-profile $\br$}} \sum_{n_1+\ldots+n_\ell=\ell, n_i\geq 0} \ind_{r(\bn)=\br} \frac{1}{\binom{\ell}{\sum_{i=1}^\ell r_i}}  (-1)^{\sum_{i=1}^\ell r_i} \prod_{i=1}^\ell \frac{(\theta^{i(i-1)/2})^{r_i}}{i!^{r_i}}
\\
&=\sum_{\text{$\ell$-profile $\br$}} \frac{1}{\binom{\ell}{\ell-\sum_{i=1}^\ell r_i}}  \prod_{i=1}^\ell \frac{(-1)^{r_i} (\theta^{i(i-1)/2})^{r_i}}{i!^{r_i}} \sum_{n_1+\ldots+n_\ell=\ell, n_i\geq 0} \ind_{r(\bn)=\br}
\\
&=\sum_{\text{$\ell$-profile $\br$}} \frac{1}{\binom{\ell}{\ell-\sum_{i=1}^\ell r_i}}  \prod_{i=1}^\ell \frac{(-1)^{r_i} (\theta^{i(i-1)/2})^{r_i}}{i!^{r_i}} \cdot w_\ell(\br)
\\
&= \sum_{\text{$\ell$-profile $\br$}} \frac{(\ell-\sum_{i=1}^\ell r_i)!(\sum_{i=1}^\ell r_i)! }{\ell!}   \prod_{i=1}^\ell \frac{(-1)^{r_i} (\theta^{i(i-1)/2})^{r_i}}{i!^{r_i}} \cdot \frac{\ell!}{(\prod_{j=1}^\ell r_j! ) \cdot (\ell-\sum_{i=1}^\ell r_i)!}
\\
&= \sum_{\text{$\ell$-profile $\br$}}  \binom{\sum_{i=1}^\ell r_i}{r_1,\ldots,r_\ell}  \prod_{i=1}^\ell \left[\frac{ -\theta^{i(i-1)/2}}{i!} \right]^{r_i},
 \end{align*}
 where we used \eqref{eqn:numberwkbr} in the sixth step.
   \end{proof}



\section{The blue region} \label{sec:blue}

In this section, we deal with the case that either $\theta\in[0,1]$ and $a\geq 0$ or $\theta\in[-1,0]$ and $a\in[-\theta,-\frac{1}{\theta}]$.

\subsection{Generating functions}

We start with a formula for the generating function of the $(p_n^a(\theta))$.

\begin{lem}[Equation~\eqref{eqn:genfunctiongenaprima} of Theorem~\ref{thm:main}]
\label{lem:formula_pna_theta>0}
Let either $\theta\in[0,1]$ and $a\geq 0$ or $\theta\in[-1,0]$ and $a\in[-\theta,-\frac{1}{\theta}]$.
Then
\begin{equation} \label{eqn:genfunctiongena}
\sum_{n=0}^\infty  p_{n}^a(\theta) z^n
= \frac{E\left(\theta,\frac{a z}{1+a}\right) -E\left(\theta,\frac{-z}{1+a}\right) }{z  E\left( \theta,\frac{-z}{1+a}\right)}.
\end{equation}
    \end{lem}




\begin{proof}
Let us start from the integral expression of the persistence probability:
\begin{align*}
    p_{n}^a(\theta) = \frac{1}{(1+a)^{n+1}} \int_{-a}^{1} \int_{\theta x_1}^{1} \ldots \int_{\theta x_{n}}^{1} 1 \dd x_{n+1}\ldots \dd x_2 \dd x_1.
\end{align*}
This representation holds because for any $x\in[-a,1]$ we have $\theta x\in[-a,1]$, due to the assumptions that either $\theta\in[0,1]$ and $a\geq 0$ or $\theta\in[-1,0]$ and $a\in[-\theta,-\frac{1}{\theta}]$.


We use the notation $H_k(x):=\frac{x^k}{k!}$ and recall that $H_{k+1}'=H_k$. We have
\begin{align*}
     p_{n}^a(\theta) 
     &= \frac{1}{(1+a)^{n+1}}\int_{-a}^{1}\int_{\theta x_1}^{1}\ldots \int_{\theta x_{n-1}}^{1}(1-\theta x_n)\dd x_{n}\ldots \dd x_2 \dd x_1\\
    &=\frac{1}{(1+a)} \,p_{n-1}^a(\theta)-\frac{1}{(1+a)^{n+1}}\,\theta\int_{-a}^{1}\int_{\theta x_1}^{1}\ldots \int_{\theta x_{n-1}}^{1}H_1(x_n)\dd x_{n} \ldots \dd x_2 \dd x_1\\
    &=\frac{1}{(1+a)}p_{n-1}^a(\theta)-\theta H_2(1)\frac{1}{(1+a)^{2}} p_{n-2}^a(\theta)
    \\
    &~~~~
    +\frac{1}{(1+a)^{n+1}}\theta^{1+2}\int_{-a}^{1}\int_{\theta x_1}^{1} \ldots \int_{\theta x_{n-2}}^{1}H_2(x_{n-1})\dd x_{n-1}\ldots \dd x_2 \dd x_1\\
    &=\sum_{k=1}^{n}\frac{(-1)^{k-1}}{k!}\theta^{\frac{k(k-1)}{2}} \frac{1}{(1+a)^{k}} \, p_{n-k}^a(\theta)+\frac{1}{(1+a)^{n+1}} (-1)^n \theta^{\frac{n(n+1)}{2}} \int_{-a}^1 H_n(x_1) \dd x_1
    \\
    &=\sum_{k=1}^{n}\frac{(-1)^{k-1}}{k!}\theta^{\frac{k(k-1)}{2}}\frac{1}{(1+a)^{k}}\,p_{n-k}^a(\theta)+\frac{1}{(1+a)^{n+1}}(-1)^n\theta^{\frac{n(n+1)}{2}}\frac{(1-(-a)^{n+1})}{(n+1)!},
\end{align*}
with $p_0^a(\theta)=\frac{1+a}{(1+a)^1}=1$. Let us abbreviate $\widehat P(z):=\sum_{n=0}^\infty  p_{n}^a(\theta) z^n$. Taking generating functions of the last equality, we deduce that
\begin{align*}
   \widehat P(z)-1 =& \sum_{n=0}^\infty  p_{n}^a(\theta) z^n - 1 
   \\
    =&\sum_{n=1}^\infty z^n \sum_{k=1}^{n}\frac{(-1)^{k-1}}{k!}\theta^{\frac{k(k-1)}{2}}\frac{1}{(1+a)^{k}}\,p_{n-k}^a(\theta)
    \\
    & +\sum_{n=1}^\infty z^n \frac{1}{(1+a)^{n+1}}(-1)^n\theta^{\frac{n(n+1)}{2}}\frac{(1-(-a)^{n+1})}{(n+1)!}
    \\
    = &\sum_{k=1}^\infty \frac{(-1)^{k-1}}{k!}\theta^{\frac{k(k-1)}{2}}\frac{z^k}{(1+a)^{k}} \sum_{n=k}^\infty z^{n-k} p_{n-k}^a(\theta)
    \\
    & +\sum_{n=1}^\infty  \frac{(-z)^n}{(1+a)^{n+1}}\theta^{\frac{n(n+1)}{2}}\frac{1}{(n+1)!}
    +\sum_{n=1}^\infty  \frac{z^n}{(1+a)^{n+1}}\theta^{\frac{n(n+1)}{2}}\frac{a^{n+1}}{(n+1)!}
        \\
    = & \left(1- E\left( \theta,\frac{-z}{1+a}\right)\right)\cdot \widehat P(z)
     +\sum_{n=2}^\infty \frac{(-z)^{n-1}}{(1+a)^{n}}\frac{\theta^{\frac{n(n-1)}{2}}}{n!}
    +\sum_{n=2}^\infty \frac{ z^{n-1}}{(1+a)^{n}}\theta^{\frac{n(n-1)}{2}}\frac{a^{n}}{n!}
    \\
        = & \left(1- E\left( \theta,\frac{-z}{1+a}\right)\right)\cdot \widehat P(z)
    \\
    & -z^{-1} \sum_{n=2}^\infty \frac{(-z)^{n} }{(1+a)^{n}}\frac{\theta^{\frac{n(n-1)}{2}}}{n!}+
    z^{-1} \sum_{n=2}^\infty \frac{(az)^{n} }{(1+a)^{n}}\frac{\theta^{\frac{n(n-1)}{2}}}{n!}
        \\
        = & \left(1- E\left( \theta,\frac{-z}{1+a}\right)\right)\cdot \widehat P(z)
    \\
    & -z^{-1}\left[ E\left(\theta,\frac{-z}{1+a}\right) + \frac{z}{1+a} - E\left(\theta,\frac{a z}{1+a}\right) +\frac{a z}{1+a} \right]
            \\
    = & \left(1- E\left( \theta,\frac{-z}{1+a}\right)\right)\cdot \widehat P(z)
     -z^{-1}\left[ E\left(\theta,\frac{-z}{1+a}\right)  - E\left(\theta,\frac{a z}{1+a}\right) \right] - 1.
\end{align*}
This gives \eqref{eqn:genfunctiongena}.
\end{proof}

As already pointed out in Remark~\ref{rem:mainthm} \ref{item:remarkmainpe}, the determination of the persistence exponent, that is, the exponential decay rate of $(p_n^a(\theta))$, can be deduced from Lemma~\ref{lem:formula_pna_theta>0}. In the case of the blue region, it is directly related to the first zero of the deformed exponential function studied in \cite{Sokal-talk,WaZh-18,Ku-24}. Namely, in the blue region, the persistence exponent $\lambda$ is determined by the smallest positive root $z_0$ ($\lambda:=z_0^{-1}$) of the equation 
$$
E( \theta , \tfrac{-z}{1+a} ) = 0.
$$
This is the zero of the denominator of (\ref{eqn:genfunctiongena}) with smallest modulus. The proof requires showing that the numerator does not vanish at $z_0$, so that the singularity of the generating function with smallest modulus is given by $z_0$. Then a combination of Theorem 2.1 in \cite{AuMuZe-21}, which ensures the existence of the persistence exponent, and Pringsheim's theorem  identifies the reciprocal of the persistence exponent with the first positive singularity.  In fact, with more work and a deeper analysis one can obtain the strong order of the persistence probabilities, i.e.\ $p_n^a(\theta) \sim c \lambda^n$, from Lemma~\ref{lem:formula_pna_theta>0}. We refer to \cite{AurzadaWorf2026conditioned} for the precise results.




\subsection{Representation in terms of Mallows-Riordan polynomials}
The next corollary shows that the persistence probabilities in the special case $a=-\theta$ can be expressed in terms of Mallows-Riordan polynomials \cite{MaRi-68}.
Define the Mallows-Riordan polynomials as the polynomials $J_n(\theta)$, $n\geq 1$, by
\begin{equation}
\label{eq:def_MR}
     \log\left(\sum_{n= 0}^\infty \theta^{n(n-1)/2}\frac{z^n}{n!} \right)=: \sum_{n=1}^\infty(\theta-1)^{n-1}J_n(\theta)\frac{z^n}{n!}.
\end{equation}
Differentiating \eqref{eq:def_MR} with respect to $z$, we obtain
\begin{equation} \label{eqn:mallowsriordanquot}
    \sum_{n=0}^\infty\frac{J_{n+1}(\theta)}{n!}z^n = \frac{\sum_{n=0}^\infty\frac{\theta^{n(n+1)/2}}{n!}(-\frac{z}{1-\theta})^n}{\sum_{n=0}^\infty\frac{\theta^{n(n-1)/2}}{n!}(-\frac{z}{1-\theta})^n},
\end{equation}
cf.\ formula (20) in \cite{AlBoRaSi-23}.

\begin{cor} \label{cor:relationtomallowsriordan1}
For all $n\geq 0$ and $\theta\in[-1,0]$, $p_n^{-\theta}(\theta)=\frac{J_{n+2}(\theta)}{(n+1)!}$.
\end{cor}

\begin{proof}
Using \eqref{eqn:genfunctiongena}, we find
\begin{equation}
\label{eq:case_MR}
    \sum_{n=0}^\infty p_n^{-\theta}(\theta)z^n = \frac{1}{z}\left(\frac{E(\theta,-\frac{\theta z}{1-\theta})}{E(\theta,-\frac{z}{1-\theta})} -1\right)= \frac{1}{z}\left(\frac{\sum_{n=0}^\infty\frac{\theta^{n(n-1)/2}}{n!}(-\frac{\theta z}{1-\theta})^n}{\sum_{n=0}^\infty\frac{\theta^{n(n-1)/2}}{n!}(-\frac{z}{1-\theta})^n} -1\right).
\end{equation}
Shifting the sum and comparing coefficients in \eqref{eqn:mallowsriordanquot} and \eqref{eq:case_MR}, we can conclude.
\end{proof}

Let us introduce the process $(W_n)$ by $W_0:=0$ and $W_n:=X_n+\theta W_{n-1}$, where the $(X_i)$ are i.i.d.\ uniform on $[-a,1]$. The process $(W_n)$ is referred to as an autoregressive process of order $1$ with uniform innovations. It is shown in Proposition~5.8 of \cite{AlBoRaSi-23} that
for $\theta \in [-1,\frac{a}{1+a}]$,
\begin{equation*}
\mathbb P(W_1\geq 0,\ldots,W_{n+1}\geq 0)=\frac{1}{(1+a)^{n+1}}\frac{J_{n+2}(\theta)}{(n+1)!}.
\end{equation*}
We do not have an intuitive explanation for the connection between the above formula (applying to autoregressive processes, called AR(1))\ and that of Corollary~\ref{cor:relationtomallowsriordan1} (which holds for MA(1) processes), other than through direct computation.

\subsection{Combinatorial representation}
The goal of this subsection is to represent the persistence probabilities in terms of the combinatorial quantities from Lemma~\ref{lem:comblem}.

Let us start with the case $a=0$ and $\theta\in[0,1]$, which is particularly simple.

\begin{cor}  \label{cor:combinatorialformulaa0theta01}  
 Let $\theta\in[0,1]$ and $a=0$. Then we have:
\begin{equation*}
p_{\ell-1}^0(\theta) = (-1)^\ell \phi_\ell(\theta)=(-1)^\ell\sum_{\text{$\ell$-profile $\br$}}  \binom{\sum_{i=1}^\ell r_i}{r_1,\ldots,r_\ell} \prod_{i=1}^\ell \left[ \left(  \frac{-\theta^{i(i-1)/2}}{i!} \right)^{r_i} \right],\qquad \ell\geq 1.
\end{equation*}
\end{cor}

\begin{proof}
    The statement follows immediately from Lemma~\ref{lem:comblem} applied to formula \eqref{eqn:genfunctiongena}, which     in the case $a=0$ becomes
\begin{equation*}
    \sum_{n=0}^\infty  p_{n}^0(\theta) z^n
= \frac{1}{z}\left(\frac{1}{E\left( \theta,-z\right)}-1\right) = \frac{1}{z}\left(\sum_{n=0}^\infty \phi_n(\theta)(-z)^n -1\right) = \sum_{\ell=0}^\infty (-1)^{\ell+1}\phi_{\ell+1}(\theta)z^\ell ,
\end{equation*}
so that the proof is complete.
\end{proof}

From Corollary~\ref{cor:combinatorialformulaa0theta01}, we can confirm that $p_{\ell-1}^0(0)=1$ for any $\ell$ (the responsible $\ell$-profile is $(\ell,0,\ldots,0)$). Subsequently, we can show that $p^0_{\ell-1}(\theta)=1-\frac{\ell-1}{2} \theta+\frac{(\ell-2)(\ell-3)}{8} \theta^2 +O(\theta^3)$ for $\theta\to 0$ (for the linear term, the responsible $\ell$-profile is $(\ell-2,1,0,\ldots,0)$, for the quadratic term $(\ell-4,2,0,\ldots,0)$). We can also obtain the monomial with the largest exponent, namely $(-1)^{\ell+1} \theta^{\ell(\ell-1)/2} \frac{1}{\ell!}$ (the responsible $\ell$-profile is $(0,\ldots,0,1)$).

On the other hand, it is not obvious to us why at $\theta=1$ we get
$$
\frac{1}{\ell!} = p_{\ell-1}^0(1) = (-1)^\ell\sum_{\text{$\ell$-profile $\br$}}  \binom{\sum_{i=1}^\ell r_i}{r_1,\ldots,r_\ell} \prod_{i=1}^\ell \left[ \left(  \frac{-1}{i!} \right)^{r_i} \right].
$$
This follows from evaluating \eqref{eqn:genfunctiongena} and combining with Corollary~\ref{cor:combinatorialformulaa0theta01}, but it is not clear how to obtain this directly. Nonetheless, it seems to be an interesting and non-trivial formula.

We can extend the combinatorial representation in Corollary~\ref{cor:combinatorialformulaa0theta01} to the whole blue region.

\begin{cor} Let $\theta\in[0,1]$ and $a\geq 0$ or $\theta\in[-1,0]$ and $a\in[-\theta,-1/\theta]$.
Then
\begin{equation*} 
p_{n}^a(\theta) =  \left(\frac{-1}{1+a}\right)^{n+1} ~\sum_{k=0}^{n+1} \phi_{n+1-k}(\theta) \frac{\theta^{\frac{k(k-1)}{2}}}{k!} (-a)^k,
\end{equation*}
where the $(\phi_\ell)$ are defined in \eqref{eqn:defphi}.
\label{cor:recurrencezeropositivethetaless1}
\end{cor}



\begin{proof} 
We have from \eqref{eqn:genfunctiongena}
\begin{equation}
    \sum_{n=0}^\infty p_n^a(\theta) z^n  = \frac{1}{z}\left[ \frac{E(\theta,\frac{az}{1+a})}{E(\theta,\frac{-z}{1+a})}-1\right]. \label{eqn:altproofrepre12}
\end{equation}
Using the definition of the deformed exponential function and Lemma~\ref{lem:comblem}, the Cauchy product formula gives
\begin{equation*}
    \frac{E(\theta,\frac{az}{1+a})}{E(\theta,\frac{-z}{1+a})} = \sum_{n=0}^\infty \sum_{k=0}^n \frac{\theta^{k(k-1)/2}}{k!} \frac{(-1)^{n-k} a^k}{(1+a)^n} \phi_{n-k}(\theta) \cdot z^{n}.
\end{equation*}
The term $-1$ in (\ref{eqn:altproofrepre12}) corresponds to eliminating the $n=0$ term. Division by $z$ in (\ref{eqn:altproofrepre12}) is an index shift by $1$. This shows
\begin{align*}
    \sum_{n=0}^\infty p_n^a(\theta) z^n & = \sum_{n=0}^\infty \sum_{k=0}^{n+1} \frac{\theta^{k(k-1)/2}}{k!} \frac{(-1)^{n+1-k} a^k}{(1+a)^{n+1}} \phi_{n+1-k}(\theta) \cdot z^{n}.
\end{align*}
Comparing coefficients shows the claim.
\end{proof}

We can also extend this by duality to the case $\theta\geq 1$ and to the case $\theta<-1$ with $-1/\theta \leq a\leq -\theta$.

\begin{cor} 
\label{cor:largethetatosmallthetaa} 
Let $(\phi_\ell)$ be as defined in \eqref{eqn:defphi}. Set $\bar\theta:=1/\theta$.
\begin{enumerate}[label={\rm (\alph{*})},ref={\rm (\alph{*})}]
    \item\label{ita}
Let $\theta\geq 1$ and $a\geq 0$. Then
\begin{equation*}
p_n^a(\theta)
=
\frac{1}{(1+a)^{n+1}} \sum_{k=0}^{n+1} \phi_{n+1-k}(\bar\theta)(-a)^{n+1-k}\frac{\bar\theta^{k(k-1)/2}}{k!}.
\end{equation*} 
\item\label{itb}
Let $\theta<-1$ with $-1/\theta \leq a\leq -\theta$. Then
\begin{equation*} 
p_n^a(\theta)
=  \left(\frac{-1}{1+a}\right)^{n+1} ~\sum_{k=0}^{n+1} \phi_{n+1-k}(\bar\theta) \frac{\bar\theta^{\frac{k(k-1)}{2}}}{k!} (-a)^k.
\end{equation*} 
\end{enumerate}
\end{cor}

\begin{proof} To prove the two parts, one can use the dualities \eqref{eq:duality_theta/pos} and \eqref{eq:duality_theta/neg}, respectively, together with Corollary~\ref{cor:recurrencezeropositivethetaless1}.
\end{proof}

\section{The green region}
\label{sec:green}

In this section, we deal with the case $\theta\in[-1,0]$ and $a\geq-1/\theta$. The idea is to reduce this case to the borderline $\theta\in[-1,0]$ and $a=-1/\theta$ by conditioning the random variables to be $\geq 1/\theta$. We can then apply the formula from the blue region which also covers this borderline.


\begin{lem}[Equation~\eqref{eqn:greengfprime} of Theorem~\ref{thm:main}] \label{lem:greenlemgf} Let $\theta\in[-1,0]$ and $a\geq-1/\theta$. Then
\begin{equation} \label{eqn:greengf}
 \sum_{n=0}^\infty p_n^a(\theta) z^n = \frac{\theta a+1}{\theta(1+a)} + \frac{E(\theta,\frac{-z}{\theta(1+a)})-E(\theta,\frac{-z}{1+a})}{z E(\theta,\frac{-z}{1+a})}.
     \end{equation}
\end{lem}

\begin{proof}
We first observe that for $n\geq 1$,
\begin{equation*}
    p_n^a(\theta) = \P( \theta X_1 \leq X_2,\theta X_2 \leq X_3, \ldots, \theta X_n \leq X_{n+1} ,X_1\geq \tfrac{1}{\theta},\ldots,X_{n+1}\geq \tfrac{1}{\theta}).
\end{equation*}
Indeed, since $\theta X_i\leq X_{i+1}\leq 1$, we get $X_i\geq \frac{1}{\theta}$ for all $1\leq i\leq n$. Further, we have $X_{n+1}\geq \theta X_n\geq \theta\geq \frac{1}{\theta}$. As a consequence, conditioning on $X_i\geq \frac{1}{\theta}$, we rewrite
\begin{equation*}
    p_n^a(\theta) = \P( \theta \widetilde X_1 \leq \widetilde X_2,\theta \widetilde X_2 \leq \widetilde X_3, \ldots, \theta \widetilde X_n \leq \widetilde X_{n+1})\cdot \nu^{n+1},
\end{equation*}
where the $\widetilde X_i$ are uniform on $[\frac{1}{\theta},1]$ and
$
    \nu = \mathbb P(X_i\geq \tfrac{1}{\theta}) = \frac{1-\frac{1}{\theta}}{1+a}$.
As it turns out, the case $a=-\frac{1}{\theta}$ is a boundary case of the blue region applicable to the $\widetilde X_i$, so we can continue the computations as follows:
\begin{equation*}
    \sum_{n=0}^\infty p_n^a(\theta) z^n = 1 -\nu +\nu \sum_{n=0}^\infty p_n^{-\frac{1}{\theta}}(\theta)(\nu z)^n,
\end{equation*}
where, using Lemma~\ref{lem:formula_pna_theta>0}, 
\begin{equation*} 
\sum_{n=0}^\infty  p_{n}^{-\frac{1}{\theta}}(\theta) (\nu z)^n
= \frac{E\left(\theta,\frac{-\nu z}{\theta(1-\frac{1}{\theta})}\right) -E\left(\theta,\frac{-\nu z}{1-\frac{1}{\theta}}\right) }{\nu z  E\left(\theta,\frac{-\nu z}{1-\frac{1}{\theta}}\right)}.
\end{equation*}
Combining the last two displays exactly corresponds to \eqref{eqn:greengf}.
\end{proof}

Analogous to the discussion after Lemma~\ref{lem:formula_pna_theta>0}, the persistence exponent $\lambda$ (i.e.\ the exponential decay rate of the $(p_n^a(\theta))$)  can be obtained as the reciprocal $\lambda:=z_0^{-1}$ of the smallest positive root $z_0$ of the equation
$$
E(\theta,\tfrac{-z}{1+a}) = 0.
$$





We can also prove a combinatorial representation of the persistence probabilities in terms of the quantities $\phi_\ell$ from \eqref{eqn:defphi}.

\begin{cor} For $\theta\in[-1,0]$ and $a\geq -1/\theta$ we have $p_0^a(\theta)=1$ and, for $n\geq 1$,
$$
p_n^a(\theta)= \frac{(-1)^{n+1}}{(1+a)^{n+1}}\, \sum_{k=0}^{n+1}  \frac{\theta^{k(k-3)/2}}{k!} \phi_{n+1-k}(\theta) ,
$$
where the $(\phi_\ell)$ are defined in \eqref{eqn:defphi}.
\label{cor:greencombrepresentation}
\end{cor}

\begin{proof}
By Lemma~\ref{lem:greenlemgf} and Lemma~\ref{lem:comblem},
\begin{align*}
  \sum_{n=0}^\infty p_n^a(\theta) z^n
 &= \frac{\theta a+1}{\theta(1+a)}
 + z^{-1}\left(\frac{E(\theta,\frac{-\bar\theta z}{1+a})  }{ E(\theta,\frac{-z}{1+a})}- 1\right)
 \\
 &= \frac{\theta a+1}{\theta(1+a)}
 + z^{-1}\left( \sum_{k=0}^\infty \frac{\theta^{k(k-1)/2}}{k!}\left(\tfrac{-\bar\theta z}{1+a}\right)^k\cdot \sum_{n=0}^\infty \phi_n(\theta) \left( \tfrac{-z}{1+a}\right)^n  - 1\right)
  \\
 &=\frac{\theta a+1}{\theta(1+a)}
 + z^{-1}\left( \sum_{k=0}^\infty \frac{\theta^{k(k-1)/2}}{k!}\left(\tfrac{-\bar\theta z}{1+a}\right)^k\cdot \sum_{n=k}^\infty \phi_{n-k}(\theta) \left( \tfrac{-1}{1+a}\right)^{n-k} z^{n-k}  - 1\right)
   \\
 &= \frac{\theta a+1}{\theta(1+a)}
 + z^{-1}\sum_{n=0}^\infty \left( \sum_{k=0}^n  \frac{\theta^{k(k-1)/2}}{k!}\bar\theta^k \phi_{n-k}(\theta) \tfrac{(-1)^n}{(1+a)^n} \cdot z^n  - 1\right)
    \\
 &= \frac{\theta a+1}{\theta(1+a)}
 + \sum_{n=1}^\infty   \sum_{k=0}^n  \frac{\theta^{k(k-3)/2}}{k!} \phi_{n-k}(\theta) \tfrac{(-1)^n}{(1+a)^n} \cdot z^{n-1}. 
\end{align*}
From here, we can read off the $p_n^a(\theta)$. 
\end{proof}

\section{The yellow region}
\label{sec:yellow}
In this section, we deal with the case $a\in[0,1]$ and $\theta\in[-1,-a]$. The idea is to relate the generating function of the persistence probabilities in this region to the one in the green region via the duality \eqref{eqn:hiddenduality2pic}.


\begin{lem}[Equation~\eqref{eq:formuma_pn_a<1_-1<theta<0_bisprime} of Theorem~\ref{thm:main}] \label{lem:a01mu}
If $a\in[0,1]$ and $\theta\in[-1,-a]$, then with $\mu :=\frac{1+\frac{a}{\theta}}{1+a} $
\begin{equation}
    \label{eq:formuma_pn_a<1_-1<theta<0_bis}
    \sum_{n=0}^\infty p_n^a(\theta)z^n=
    \frac{1}{z}\left(\frac{1}{\frac{E(\theta, \frac{az}{\theta(1 + a)})}{E(\theta, \frac{az}{1 + a})}-\mu z}-1\right).
\end{equation}
\end{lem}

\begin{proof}
This will follow directly from the duality stated in Corollary~\ref{cor:second_hidden_duality}, combined with the expression provided in Lemma~\ref{lem:greenlemgf} for the green region. Indeed, denoting by $\widehat P_{a,\theta}(z)$ the formula \eqref{eqn:greengf} in the green region, we have, using Corollary~\ref{cor:second_hidden_duality},
\begin{equation*}
    \sum_{n=0}^\infty p_n^a(\theta)z^n= \frac{\widehat P_{1/a,\theta}(-z)}{1-z \, \widehat P_{1/a,\theta}(-z)}= \frac{1}{z}\left(\frac{1/z}{1/z-\widehat P_{1/a,\theta}(-z)}-1 \right). 
\end{equation*}
On the other hand, a computation starting from \eqref{eqn:greengf} shows that
\begin{equation*}
    \widehat P_{1/a,\theta}(-z) = \mu -\frac{1}{z}\left(\frac{E(\theta, \frac{az}{\theta(1 + a)})}{E(\theta, \frac{az}{1 + a})}-1\right),
\end{equation*}
and we easily conclude \eqref{eq:formuma_pn_a<1_-1<theta<0_bis}.
\end{proof}

Analogous to the discussion after Lemma~\ref{lem:formula_pna_theta>0}, the persistence exponent $\lambda$ (i.e.\ the exponential decay rate of the $(p_n^a(\theta))$) can be obtained as the reciprocal $\lambda:=z_0^{-1}$ of the smallest positive root $z_0$ of the equation
$$
\frac{E(\theta,\frac{az}{\theta(1+a)})}{E(\theta,\frac{az}{1+a})} = \mu z.
$$




In contrast to the blue and green regions, we did not find an explicit expression of the $(p_n^a(\theta))$ in terms of the combinatorial quantities $(\phi_\ell(\theta))$ from Lemma~\ref{lem:comblem}.

\section{The orange region}
\label{sec:orange}

First, we deal with the case $-b:=a\leq 0$ and $1\leq \theta\leq 1/b$. Note that, contrary to the other regions, we start with the computation for $\theta>1$ and then deduce the result for $\theta\in[0,1]$. This is due to the fact that the probabilities are much simpler for $\theta>1$.

In fact, for $b>0$, the persistence probabilities become zero at some point. Even more, the remaining positive persistence probabilities can be written down in a very simple form. The case $a=b=0$ is a special case of this computation.

\begin{lem} \label{lem:grey}  Let $a\in(-1,0]$ and $1\leq \theta\leq -1/a$. We set $b:=-a$. Then
\begin{align} \label{eqn:greysimpleclaim}
    p_n^{a}(\theta) = \frac{1}{(1-b)^{n+1}} \,\frac{\theta^{(n+1)n/2}}{(n+1)!} \,\left( \frac{1}{\theta^n} - b\right)_+^{n+1},
\end{align}
where $x_+:=\max(x,0)$. In particular,
\begin{itemize}
    \item If $b>0$ then we have 
$$
p_n^{a}(\theta) = 0,\qquad \text{for all $n$ such that $\theta^n b \geq 1$}.
$$
\item 
If $b=0$ then the formula simplifies to 
\begin{align*}
    p_n^{0}(\theta) = \frac{\theta^{-(n+1)n/2}}{(n+1)!}.
\end{align*}
\end{itemize}
\end{lem}

We remark that, for $b>0$ and $\theta>1$, due to the fact that $\theta b \leq 1$ and $\theta>1$, we know that $\theta^n b\geq 1$ for all $n\geq n_0$ for some finite $n_0\geq 1$.

Further, given the simple form of the persistence probabilities, one can easily compute the generating function, e.g.\ for $b=0$:
\begin{equation} \label{eqn:b=0formula}
\sum_{n=0}^\infty  p_n^0(\theta)z^n = \frac{1}{z}(E(\theta^{-1},z)-1).
\end{equation}

Another remark is that the formula \eqref{eqn:b=0formula} can also be obtained from Lemma~\ref{lem:formula_pna_theta>0} together with the duality \eqref{eq:duality_theta/pos}: Fixing $\theta>1$ and letting $a\searrow 0$, we have
$$
\sum_{n=0}^\infty  p_n^a(\theta)z^n = \sum_{n=0}^\infty  p_n^{1/a}(1/\theta)z^n = \frac{E(\theta^{-1},\frac{z}{a+1})-E(\theta^{-1},-\frac{az}{1+a})}{zE(\theta^{-1},-\frac{az}{1+a})} \to \frac{1}{z} (E(\theta^{-1},z)-1) .
$$

Concerning the persistence exponent, we notice that the persistence probabilities in Lemma~\ref{lem:grey} are either equal to zero for large enough $n$ or decay superexponentially. The latter can actually be seen independently of the generating function computation and is due to the assumptions $a\leq 0$ and $\theta\geq 1$ which yield $p_n^a(\theta)\leq p_n^a(1) = \frac{1}{(n+1)!}$.

\begin{proof}[Proof of Lemma~\ref{lem:grey}]
Note that if $X_1>1/\theta$ then $\theta X_1 \leq X_2$ implies that $X_2>1$, which is not possible. Similarly, $X_1>1/\theta^{k-1}$ and the relations $\theta^{k-1} X_1 \leq \theta^{k-2} X_2 \leq \ldots \leq \theta X_{k-1} \leq X_k$ imply $X_k>1$.     On the other hand, if $x_k\geq b$ then $\theta x_k \geq \theta b\geq b$.

First, the above reasoning implies that $p_n^a(\theta)=0$ for $b\geq 1/\theta^n$. Thus, for the rest of the proof, we can concentrate on the case $b\leq 1/\theta^n$.
    
Taking the above considerations into account, we have
\begin{align*}
    p_n^a(\theta) & = \P( \theta X_1\leq X_2, \ldots, \theta X_n \leq X_{n+1})
    \\
    & = \frac{1}{(1-b)^{n+1}}\, \int_b^{1/\theta^n}\int_{\theta x_1}^{1/\theta^{n-1}} \ldots \int_{\theta x_{n-1}}^{1/\theta} \int_{\theta x_n}^1 1 \dd x_{n+1} \dd x_n \ldots \dd x_2 \dd x_1.
    \end{align*}
For simplicity of notation, set
\begin{align}
q_n^a(\theta) & := (1-b)^{n+1} p_n^a(\theta)
\notag \\
&  = \int_b^{1/\theta^n}\int_{\theta x_1}^{1/\theta^{n-1}} \ldots \int_{\theta x_{n-1}}^{1/\theta} \int_{\theta x_n}^1 1 \dd x_{n+1} \dd x_n \ldots \dd x_2 \dd x_1 \label{eqn:greysimpleline1}
\\ & = \int_b^{1/\theta^n} q_{n-1}^{-\theta x_1}(\theta)\, \dd x_1, \notag
    \end{align}
and it only remains to check the claim of the lemma by induction using the latter recursion. For $n=1$, we have due to \eqref{eqn:greysimpleline1}
\begin{equation*}
q_1^a(\theta)
  = \int_b^{1/\theta}\int_{\theta x_1}^{1} 1 \dd x_2 \dd x_1 =  \int_b^{1/\theta}(1-\theta x_1) \dd x_1  = \frac{1}{\theta}-b - \frac{\theta}{2}\left( \frac{1}{\theta^2} - b^2\right) = \frac{\theta}{2}\left( \frac{1}{\theta} - b\right)^2,
\end{equation*}
in accordance with \eqref{eqn:greysimpleclaim}. If \eqref{eqn:greysimpleclaim}
 is true for $n-1$, we get from the recursion (using that $\frac{1}{\theta^{n-1}}\geq \theta x_1$) that
\begin{align*}
q_n^a(\theta)
& = \int_b^{1/\theta^n} q_{n-1}^{-\theta x_1}(\theta) \dd x_1
\\
& = \int_b^{1/\theta^n} \frac{\theta^{(n-1)n/2}}{n!} \,\left( \frac{1}{\theta^{n-1}} - \theta x_1\right)^{n} \dd x_1
\\
& = \frac{\theta^{(n-1)n/2}}{n!} \, \int_b^{1/\theta^n} \theta^{n} \left( \frac{1}{\theta^{n}} - x_1\right)^{n} \dd x_1
\\
& = \frac{\theta^{(n+1)n/2}}{(n+1)!} \, \left. (-1) \left( \frac{1}{\theta^{n}} - x_1\right)^{n+1}\, \right|_{x_1=b}^{1/\theta^n}
\\
& = \frac{\theta^{(n+1)n/2}}{(n+1)!} \, \left( \frac{1}{\theta^{n}} -b\right)^{n+1},
\end{align*}
as required in the case $b\leq 1/\theta^n$.
\end{proof}




We now treat the case $\theta\in(0,1]$. 

\begin{lem}[Equation~\eqref{eqn:darkorangegfprime} of Theorem~\ref{thm:main}]  \label{cor:regionp} Let $\theta\in(0,1]$ and $0\leq -a=:b\leq \theta$. Assume that $p\geq 1$ is such that $\theta^{p+1}< b \leq \theta^p$. Then the generating function of the persistence probabilities is given by
\begin{align} \label{eqn:darkorangegf}
    \sum_{n=0}^\infty p_n^a(\theta) z^n=\frac{F(z)}{1-z F(z)},
\end{align}
    where as defined in (\ref{eqn:darkorangedefnf0}) (recall $b=-a$)
\begin{align} \label{eqn:darkorangedefnf}
    F(z):=\sum_{i=0}^{p} (-1)^i  \frac{(\theta^i-b)^{i+1}}{(i+1)!\theta^{i(i+1)/2}(1-b)^{i+1}}\, z^i = \sum_{i=0}^\infty (-1)^i  \frac{(\theta^i-b)_+^{i+1}}{(i+1)!\theta^{i(i+1)/2}(1-b)^{i+1}}\, z^i.
\end{align}
    In particular, the generating function $\sum_{n=0}^\infty p_n^a(\theta) z^n$ is a piecewise rational function.
\end{lem}

\begin{proof} The lemma follows directly from the duality \eqref{eqn:hiddendualityGF}
and Lemma~\ref{lem:grey}.
\end{proof}

We remark that the function $F$ is nothing else but the generating function of the persistence probabilities $(p_n^a(1/\theta))$ computed from \eqref{eqn:greysimpleclaim}.

Analogous to the discussion after Lemma~\ref{lem:formula_pna_theta>0}, the persistence exponent $\lambda$ (i.e.\ the exponential decay rate of the $(p_n^a(\theta))$) can be obtained as the reciprocal $\lambda:=z_0^{-1}$ of the smallest positive root $z_0$ of the equation
\begin{align} \label{eqn:pefororanged}
z F(z) & = 1,
\end{align}




We remark that for $b>0$ the equation \eqref{eqn:pefororanged} is a polynomial equation for the persistence exponent $\lambda$ (the degree of the polynomial is $p+1$ if $\theta^{p+1}\leq b < \theta^p$).

\subsection*{Acknowledgments}
The authors would like to thank both referees for their careful reading of our manuscript and the many suggestions. The authors thank Virginia Worf (Darmstadt) for discussions on the manuscript. KR is supported by the project RAWABRANCH (ANR-23-CE40-0008), funded by the French National Research Agency.

\subsection*{Data availability statement}
No new data were generated or analysed in this study.

\subsection*{Conflict of interest statement}
The authors declare that they have no conflict of interest.

\bibliographystyle{plain} 
\bibliography{persist.bib}

\end{document}